\documentclass[12pt]{amsart}         
\usepackage{latexsym,mathtools}    
\usepackage{mathrsfs}
\usepackage{epsfig,setspace} 
\usepackage{a4}
\usepackage{comment}
\usepackage{amssymb}
\usepackage{amsthm}
\usepackage{amsbsy} 
\usepackage{upgreek}
\usepackage{relsize}
\usepackage{tikz-cd}
\usepackage{ytableau}
\usepackage{amsfonts,amssymb,latexsym,amscd,amsmath,euscript,enumerate,verbatim,calc}
\allowdisplaybreaks
\usepackage{url}
\usepackage{hhline}
\usepackage{pifont}
\usepackage{relsize}
\usepackage{graphicx}    
\usepackage{subcaption}  
\usepackage{appendix}
\usepackage[backend=biber]{biblatex}  

\DeclareFieldFormat{doi}{\href{https://dx.doi.org/#1}{\textsc{doi}}}
\DeclareFieldFormat{url}{\href{#1}{\textsc{url}}}
\renewbibmacro*{doi+eprint+url}{
  \printfield{doi}
  \newunit\newblock
  \iftoggle{bbx:eprint}{
    \usebibmacro{eprint}
  }{}
  \newunit\newblock
  \iffieldundef{doi}{
    \usebibmacro{url+urldate}}
  {}
}
\renewbibmacro{in:}{}
\AtEveryBibitem{
  \ifentrytype{book}
  {\clearfield{pages}}
}

\usepackage[colorlinks=true,linkcolor=blue,citecolor=magenta]{hyperref}  

\theoremstyle{definition}

\newtheorem*{theorem*}{Theorem}
\newtheorem{theorem}{Theorem}[section]   
\newtheorem{lemma}[theorem]{Lemma}
\newtheorem{proposition}[theorem]{Proposition}

\newtheorem{problem}[theorem]{Problem}
 
\newtheorem{corollary}[theorem]{Corollary} 
\newtheorem{definition}[theorem]{Definition}

\newtheorem{example}[theorem]{Example}

\def\CC{{\mathbb C}}
\def\PP{{\mathbb P}}
\def\QQ{{\mathbb Q}}
\def\Fq{{\mathbb F}_q}

\def\mm{{\mathsf m}}

\def\he{\mathscr{H}}
\def\fl{{\mathrm{fl}}}
\def\asc{{\mathrm{asc}}}
\def\fl{{\mathrm{fl}}}
\def\DD{\mathcal{D}}
\makeatletter
\@namedef{subjclassname@2020}{\textup{2020} Mathematics Subject Classification}
\makeatother

 \title[Hessenberg varieties over finite fields]{Counting points on Hessenberg varieties over finite fields}     
 \author{Alex Abreu}
 \address{Instituto de Matemática e Estatística, Universidade Federal Fluminense, Niterói, Rio de Janeiro, Brasil.}
 \email{alexbra1@gmail.com}

 \author{Antonio Nigro}
 \address{Instituto de Matemática e Estatística, Universidade Federal Fluminense, Niterói, Rio de Janeiro, Brasil.}
 \email{antonio.nigro@gmail.com} 
 
 \author{Samrith Ram} 
\address{Indraprastha Institute of Information Technology Delhi, New Delhi, India.}
\email{samrithram@gmail.com}


\begin{document}
\begin{abstract}
  We give a counting formula in terms of modified Hall-Littlewood polynomials and the chromatic quasisymmetric function for the number of points on an arbitrary Hessenberg variety over a finite field. As a consequence, we express the Poincaré polynomials of complex Hessenberg varieties in terms of a Hall scalar product involving the symmetric functions above. We use these results to give a new proof of a combinatorial formula for the modified Hall-Littlewood polynomials.
\end{abstract}
\maketitle

\section{Hessenberg varieties}
 Let $[n]$ denote the set of the first $n$ positive integers. A \emph{Hessenberg function} is a weakly increasing function $\mm:[n]\to [n]$ such that  $m(i)\geq i$ for each $i \in [n]$. Given a field $F,$ and a linear operator $T:F^n\to F^n,$ the \emph{Hessenberg variety} $\he(\mm,T)$ is defined by
\begin{align*}
  \he(\mm,T):= \{\text{complete flags } V_1\subseteq V_2\subseteq \cdots \subseteq V_n=F^n: TV_i\subseteq V_{\mm(i)}\text{ for } i\in [n]\}.
\end{align*}

Hessenberg  varieties were originally introduced by De Mari, Procesi and Shayman~\cite{MR1043857}, and are important subvarieties of the full flag variety. In the case where $\mm(i)=i$ for $1\leq i\leq n$, the variety $\he(\mm,T)$ is called the Springer variety associated with $T$. Let $\Fq$ denote the finite field with $q$ elements. The main objective of this article is to derive an explicit formula for the number of $\Fq$-rational points on $\he(\mm,T)$ for a linear operator $T$ on $\Fq^n$. One reason for seeking such a formula is that point counting often yields insights into the cohomology of varieties over complex numbers. Another motivation for seeking such a formula comes from the work of Goresky, Kottwitz and MacPherson \cite{MR2209851,MR2040285} who proved that certain orbital integrals can be interpreted through the process of counting points on certain (generalizations of) Hessenberg varieties (also see Tsai~\cite{MR3622875} where the problem of counting points on certain étale schemes over Hessenberg varieties is considered). The question of point enumeration also appears explicitly on MathOverflow \cite{MO196537}.

Hessenberg varieties were considered in the finite-field setting by Fulman~\cite{MR1704269} who applied the Weil conjectures to obtain results on the combinatorics of descents in the symmetric group. Tymoczko \cite{MR2275912} showed that complex Hessenberg varieties are paved by affines and gave combinatorial formulas for their Poincaré polynomials; these results have implications for point counting (see Theorem~\ref{thm:sameforcandfq}). Escobar, Precup and Shareshian \cite{EPS} gave an explicit formula for the number of $\Fq$-rational points on a Hessenberg variety of codimension one. Gagnon \cite{Gagnon} considered the problem of point counting on variant nilpotent Hessenberg varieties associated to ad-nilpotent ideals. Ram \cite{ram2024lusztigvarietiesmacdonaldpolynomials} used point counts of Lusztig varieties (which are generalizations of Hessenberg varieties) to compute traces in Hecke algebras and made connections to parabolic Springer fibers. One feature of the existing results on counting $\Fq$-rational points is that they almost exclusively deal with the case where the linear operator defining the Hessenberg variety has a characteristic polynomial that splits over the base field. One of the main contributions of this paper is to give a universal formula that also accounts for the non-split cases (Theorem \ref{thm:main}). Our enumeration result also relates point counting in the general case to the combinatorics of Hall-Littlewood polynomials.

In order to state the counting formula, we associate suitably defined symmetric functions to the Hessenberg function $\mm$ and the operator $T$. For any graph $G$ with vertex set $[n]$, consider all vertex-colorings, defined as functions $\kappa:G\to \PP,$ where $\PP$ denotes the set of positive integers. The \emph{chromatic quasisymmetric function} of $G$ is a generating function in infinitely many variables $x=(x_1,x_2,\ldots)$ defined by 
\begin{align*}
  X_G(x;t):=\sum_{\kappa}t^{\asc({\kappa})}x_{\kappa},
\end{align*}
where the sum is taken over all proper vertex colorings $\kappa$ and
\begin{align*}
  \asc(\kappa):=|\{\{i,j\}\in E(G):i<j \text{ and }\kappa(i)<\kappa(j)\}|,
\end{align*}
is the number of ascents of the coloring $\kappa$. Here $x_\kappa$ denotes the monomial $x_{\kappa(1)}x_{\kappa(2)}\cdots x_{\kappa(n)}$.  Chromatic quasisymmetric functions  of graphs were introduced by Shareshian and Wachs \cite{MR3488041}, building on earlier work by Stanley \cite{MR1317387}. In fact $X_G(x;t)$ can be realized via the complex character theory of finite general linear groups. For more on this topic see Gagnon~\cite{Gagnon} who also makes connections between point counting on certain variations of nilpotent Hessenberg varieties and some results of Precup and Sommers~\cite{MR4391520}.

To each Hessenberg function, we associate an indifference graph $G(\mm)$. The graph $G(\mm)$ has vertex set $[n]$ and edge set $E(G)=\{\{i,j\}:i<j\leq \mm(i)\}$. It is well known that $X_{G(\mm)}(x;t)$ is in fact symmetric in the variables $x_i$. 

 Let $\omega$ denote the standard involution on the ring of symmetric functions which takes the Schur function $s_\lambda$ to $s_{\lambda'},$ where $\lambda'$ denotes the partition conjugate to $\lambda$. We are now ready to state our main result.
\begin{theorem}
\label{thm:main}
For each operator $T$ on $\Fq^n$, the number of points on the Hessenberg variety $\he(\mm,T)$ is given by
  $$
|\he(\mm,T)|=\langle F_T(x), \omega X_{G(\mm)}(x;q)\rangle.
$$
where $\langle \cdot,\cdot \rangle$ denotes a Hall scalar product. 
\end{theorem}
Here $F_T(x)$ is a symmetric function that can be characterized in terms of the conjugacy class invariants of $T$ and admits an explicit formula in terms of modified Hall-Littlewood polynomials (see Proposition \ref{prop:ifgf}).

\begin{corollary}
\label{cor:simple_hessenberg}
    Let $T$ be a linear operator of similarity class type (see Definition \ref{def:sct}) $\{(d_1,\lambda^1), \ldots, (d_r,\lambda^r)\}$ and let $\mm = (\mm(1),n,n\ldots, n)$ be a Hessenberg function satisfying $\mm(2)=n$.  We have
    \[
    |\he(\mm,T)| = [n-2]_q!\Big([n]_q[\mm(1)-1]_q + q^{\mm(1)-1}[n-\mm(1)]_q\sum_{\substack{1\leq i\leq r\\ d_i=1}}[\ell(\lambda^i)]_q\Big).
    \]
Here $[m]_q$ denotes the $q$-analog of the integer $m$, while $\ell(\lambda)$ denotes the number of parts of $\lambda$. In particular, if $\mm(1)=n-1$, we recover the following result of Escobar, Precup and Shareshian \cite[Proposition 2]{EPS}:
    \[
    |\he(\mm,T)| = [n-2]_q!\Big([n-2]_q[n]_q+q^{n-2}\sum_{\substack{1\leq i\leq r \\ d_i=1}}[\ell(\lambda^i)]_q 
   \Big). \]    
\end{corollary}

 The modular law (Definition~\ref{def:modular}) plays a crucial role in our elementary and combinatorial proof of Theorem \ref{thm:main}. As a corollary of Theorem \ref{thm:main} we obtain, in Section \ref{sec:main}, explicit formulas involving modified Hall-Littlewood polynomials for the Poincaré polynomials of complex Hessenberg varieties (Corollary \ref{cor:poinc}). Tymoczko \cite[Thm. 7.1]{MR2275912} has already given combinatorial formulas for these polynomials. Our result is then used to give a new proof of a theorem of Brosnan and Chow (see Corollary \ref{cor:poinreg}) on the monomial coefficients of the omega dual of the chromatic quasisymmetric function. This theorem is known to imply the fact that the Betti numbers of complex regular Hessenberg varieties form a palindromic sequence.

The paper is organized as follows. In Section \ref{sec:ifgf} we discuss the symmetric function $F_T(x)$ appearing in the proof of the main theorem and show how it can be expressed in terms of conjugacy class invariants of $T$. We give examples to show that several natural bases of symmetric functions such as the power sum, complete homogeneous and modified Hall-Littlewood polynomials arise as $F_T(x)$ for suitable choices of $T$. In Section \ref{sec:main} we prove the main theorem and discuss applications to Poincaré polynomials of complex Hessenberg varieties.  In  the Appendix we use Theorem \ref{thm:main} to derive a combinatorial formula for modified Hall-Littlewood polynomials. 

\section{Invariant flag generating functions}\label{sec:ifgf}
Given a linear operator $T$ on $\Fq^n$, and an integer partition $\lambda=(\lambda_1,\lambda_2,\ldots,\lambda_\ell)$ of $n$, a $\lambda$-flag in $\Fq^n$ is a sequence  $(0)=W_0\subseteq W_1\subseteq \cdots \subseteq W_\ell$ of subspaces of $\Fq^n$ such that $\dim W_i/W_{i-1}=\lambda_i$ for each $1\leq i\leq \ell$. Let $\mathcal{F}_\lambda(T)$ denote the collection of $\lambda$-flags of $T$-invariant subspaces of $\Fq^n$. We can now define the symmetric function associated to $T$. 

\begin{definition}
  The \emph{invariant flag generating function} of $T$ is defined by
  \begin{align*}
    F_T(x):=\sum_{\lambda \vdash n}\fl_\lambda(T)m_\lambda,
  \end{align*}
  where the sum is over all integer partitions $\lambda$ of $n$ and $\fl_\lambda(T):=|\mathcal{F}_\lambda(T)|,$ while $m_\lambda$ denotes the monomial symmetric function indexed by $\lambda.$ 
\end{definition}

The symmetric function $F_T(x)$ was introduced in \cite{ram2023subspace} in connection with a problem posed in 1992 by Bender, Coley, Robbins and Rumsey on subspace profiles. Special cases of this problem had been studied in several papers \cite{MR2831705,split,MR3093853,MR4263652,MR4349887, MR4555237,MR4682040,MR4797454, ram2023diagonal} with the general solution appearing in \cite{ram2023subspace}. One can express $F_T(x)$ explicitly in terms of \emph{modified Hall-littlewood polynomials}. Recall that the modified Hall-Littlewood polynomial $\tilde{H}_\lambda(x;t)$ is defined by
\begin{align*}
  \tilde{H}_\lambda(x;t):=\sum_{\mu}\tilde{K}_{\mu\lambda}(t)s_\mu,
\end{align*}
where $s_\mu$ denotes a Schur function and $\tilde{K}_{\mu\lambda}(t)$ is a modified Kostka-Foulkes polynomial, defined as the generating polynomial of the cocharge statistic on semistandard tableaux:
\begin{align*}
  \tilde{K}_{\mu\lambda}(t)=\sum_{\mathcal{T} \in {\rm SSYT(\mu,\lambda)}}t^{\rm cocharge{(\mathcal{T})}}.
\end{align*}
Modified Hall-Littlewood polynomials are very combinatorial in nature and occur as Frobenius images of suitably defined graded vector spaces. When $q$ is a prime power, the specialization $\tilde{K}_{\lambda\mu}(q)$ equals the value taken by the character of the irreducible unipotent ${\rm GL}_n(\Fq)$-representation indexed by $\lambda$ on a unipotent element with Jordan form partition $\mu$ (see Lusztig \cite[Eq. 2.2]{MR641425}). When $\lambda=(n)$ has a single part, we have $\tilde{H}_\lambda(x;t)=h_n$, the complete homogeneous symmetric function.

One can derive an expression  for $F_T(x)$ involving modified Hall-Littlewood polynomials in terms of the conjugacy class data of $T$. Recall that the action of $T$ on $\Fq^n$ makes $\Fq^n$ into an $\Fq[t]$-module which is isomorphic to a direct sum
\begin{align*}
  \Fq^n\simeq \bigoplus_{i=1}^r\bigoplus_{j=1}^{\ell_i}\frac{\Fq[t]}{( f_i^{\lambda_{i,j}})},
\end{align*}
where $f_i(t)\in \Fq[t]$ are distinct monic irreducible polynomials and $\lambda^i=(\lambda_{i,1}, \lambda_{i,2},\ldots)$ is an integer partition for $1\leq i\leq r.$ The conjugacy class of $T$ is uniquely determined by the polynomials $f_i$ and the partitions $\lambda^i(1\leq i\leq r)$. 
\begin{definition}\label{def:sct}
    Writing $d_i$ for the degree of $f_i$, the \emph{similarity class type} of $T$ is defined to be the multiset $\tau=\{(d_1,\lambda^1),\ldots,(d_r,\lambda^r)\}$.
\end{definition}
The notion of similarity class type goes back to the work of Green \cite{MR72878} on the characters of the finite general linear groups. To each similarity class type $\tau$, we associate a symmetric function
\begin{align*}
 F_\tau(x;t):= \prod_{i=1}^r \tilde{H}_{\lambda^i}(x_1^{d_i},x_2^{d_i},\ldots;t^{d_i}) = \prod_{i=1}^r p_{d_i}[\tilde{H}_{\lambda^i}(x;t)],
\end{align*}
where the square brackets indicate plethystic substitution. We have the following result  \cite[Prop. 2.11]{ram2023subspace}.
\begin{proposition}\label{prop:ifgf}
If $T$ is an operator over $\Fq$ of similarity class type $\tau$, then
\begin{equation*}
  F_T(x)=F_\tau(x;q).
\end{equation*}
\end{proposition}

Several well-studied bases of the ring of symmetric functions arise as invariant flag generating functions of linear operators (see \cite[Sec. 2]{ram2023subspace}). We give a few examples here.
\begin{example}
  An operator $T$ is \emph{triangulable} if there exists a basis of $\Fq^n$ with respect to which the matrix of $T$ is upper triangular (equivalently, its characteristic polynomial splits completely in $\Fq$). Let $a_i(1\leq i\leq r)$ be the distinct eigenvalues of $T$. The \emph{Jordan type} of $T$ is the multiset $\Lambda=\{\lambda^i\}_{1\leq i\leq r}$, where $\lambda^i$ is the integer partition corresponding to $a_i$ in the Jordan canonical form of any matrix for $T$. In this case, $F_T(x)=\prod_{i=1}^r \tilde{H}_{\lambda^i}(x;q)$. 
\end{example}

We record below two specific instances of the above example that are worth noting. 

\begin{example}\label{eg:nilpotent}
  If $T$ is nilpotent over $\Fq$ with Jordan form partition $\lambda,$ then its invariant flag generating function is given by $ F_T(x)=\tilde{H}_\lambda(x;q).$ 
\end{example}

\begin{example}\label{eg:regularsplit}
  An operator $T$ is \emph{regular split} if the corresponding $\Fq[t]$ module on $\Fq^n$ is isomorphic to a direct sum
  $$
\bigoplus_{i=1}^r \frac{\Fq[t]}{(x-a_i)^{\alpha_i}},
$$
for distinct elements $a_i\in \Fq$  and positive integers $\alpha_i(1\leq i\leq r)$. Such an operator is characterized by the fact that its minimal and characteristic polynomials coincide and each is a product of linear factors over $\Fq.$ In this case we have $F_T=h_\alpha$, the complete homogeneous symmetric function indexed by the composition $\alpha$ whose parts are the $\alpha_i(1\leq i\leq r)$.
\end{example}

\begin{example}\label{eg:regularsemisimple}
  A linear operator $T$ on $\Fq^n$ is \emph{regular semisimple} if its characteristic polynomial is a product of distinct irreducible polynomials over $\Fq.$ In this case the corresponding $\Fq[t]$ module is a direct sum
  $$
\bigoplus_{i=1}^r \frac{\Fq[t]}{(f_i)},
$$
for distinct irreducible polynomials $f_i\in \Fq[t].$ If we write $\alpha_i=\deg f_i$ for $1\leq i\leq r$, then $F_T(x)=\prod_{i=1}^r p_{\alpha_i}=p_\alpha$, the power sum symmetric function indexed by the composition $\alpha=(\alpha_1,\alpha_2,\ldots)$.
\end{example}

\section{The main result}\label{sec:main}
We begin with a discussion of the modular law \cite[Defn. 2.1]{MR4199388} which plays a pivotal role in the proof of the main theorem. For convenience we sometimes denote a Hessenberg function $\mm:[n]\to [n]$ by the tuple $(\mm(1),\ldots,\mm(n))$. A Dyck path of size $n$ in the Euclidean plane is lattice path from $(0,0)$ to $(n,n)$ that uses only north steps $(0,1)$ and east steps $(1,0)$ and never goes below the line $x=y$. To each Hessenberg function $\mm:[n]\to [n]$ we associate a unique Dyck path of size $n$ by requiring that there are precisely $\mm(i)$ north steps before the $i$-th east step (see Figure \ref{fig:hessenberg}). It will be convenient to identify Hessenberg functions with their corresponding Dyck paths. Denote by $\DD$ and $\DD_n$ the collection of all Dyck paths and Dyck paths of size $n$. 

\begin{figure}
  \centering
  \includegraphics{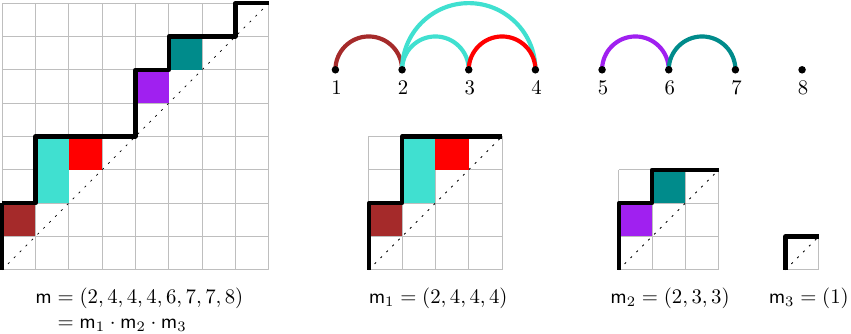}   
  \caption{Left: A Hessenberg function $\mm$ with the corresponding Dyck path; Right: The indifference graph $G(\mm)$ and the irreducible components $\mm_1,\mm_2,\mm_3$ of $\mm$.}
  \label{fig:hessenberg}
\end{figure}

Given Dyck paths $\mm_1$ and $\mm_2$ of sizes $n_1$ and $n_2$, their concatenation, denoted $\mm_1\cdot \mm_2$, is the path of size $n_1+n_2$ obtained by traversing all steps in $\mm_1$ followed by those in $\mm_2$.  A Dyck path is \emph{irreducible} if it cannot be written as a concatenation of nontrivial Dyck paths. It is easily seen that every Dyck path $\mm\in \DD$ decomposes uniquely as a concatenation of irreducible Dyck paths, called the irreducible components of $\mm$ (see Figure \ref{fig:hessenberg}).  The \emph{complete} Hessenberg function $k_n$ is the unique Hessenberg function in $\DD_n$ defined by $k_n(1)=n$.

\begin{definition}\label{def:modular}
Let $\mathcal{A}$ be a $\QQ(t)$-algebra. A function $f : \DD_n \to \mathcal{A}$ satisfies the \emph{modular law} if
\begin{equation*}
(1 + t) f(\mm_1) = tf(\mm_0) + f(\mm_2),
\end{equation*}
whenever one of the following conditions is satisfied.

\begin{enumerate}
    \item There exists $i \in [n-1]$ such that $\mm_1(i-1) < \mm_1(i) < \mm_1(i+1)$ and $\mm_1(\mm_1(i)) = \mm_1(\mm_1(i)+1)$. In addition, for each $k \in \{0, 2\}$, we have $\mm_k(j) = \mm_1(j)$ for every $j \neq i$  while $\mm_k(i) = \mm_1(i) - 1 + k$.
    
    \item There exists $i \in [n-1]$ such that $\mm_1(i+1) = \mm_1(i) + 1$ and $\mm_1^{-1}(i) = \emptyset$. Moreover, for each $k \in \{0, 2\}$, we have $\mm_k(j) := \mm_1(j)$ for every $j \neq i, i+1$, while $\mm_0(i) = \mm_0(i+1) = \mm_1(i)$ and $\mm_2(i) = \mm_2(i+1) = \mm_1(i+1)$.
\end{enumerate}
\end{definition}

For a Hessenberg function $\mm:[n]\to [n]$, write
\begin{align*}
  X_{G(\mm)}(x;t)=\sum_{\lambda \vdash n}a_{t,\lambda}(\mm)e_\lambda.
\end{align*}
Denote by $[n]_t:=1+t+\cdots+t^{n-1}$ the $t$-analog of the integer $n$ and write $[n]_t!$ for the product $\prod_{j=1}^n[j]_t.$ We require the following result \cite[Cor. 3.2]{MR4199388}.
\begin{proposition}\label{prop:base}
  Suppose $\mathcal{A}$ is a $\QQ(t)$-algebra and let $f:\DD\to \mathcal{A}$ be a function satisfying the modular law. Then
  \begin{align*}
    f(\mm)=\sum_{\lambda \vdash n}\frac{a_{t,\lambda}(\mm)}{[\lambda]_t!}f(k_\lambda),
  \end{align*}
  where $k_\lambda=k_{\lambda_1}\cdot k_{\lambda_2}\cdots$ and $[\lambda]_t!$ denotes the product $\prod_{j\geq 1}[\lambda_j]_t!$.
\end{proposition}

Given $\mm$ and $T$, define
\begin{equation*}
  f(\mm,T):=|\he(\mm,T)|,
\end{equation*}
the number of points on the corresponding Hessenberg variety.

\begin{lemma}\label{lem:modular}
 If $\mm_0,\mm_1,\mm_2$ are Hessenberg functions satisfying the hypotheses of the modular law and $T$ is a linear operator over $\Fq$, then
  \begin{align*}
    (1+q)f(\mm_1,T)=qf(\mm_0,T)+f(\mm_2,T).
  \end{align*}
\end{lemma}
\begin{proof}
We follow \cite[Example 3.5]{AN_haiman}. Assume that $\mm_0,\mm_1,\mm_2$ satisfy condition (1) in Definition \ref{def:modular} for some $i$ and set $\ell=\mm_1(i)$. Consider, for $r=0,1,2$, the forgetful maps
\[
\alpha_r\colon  \he(\mm_r,T)\to  \{V_1\subset \ldots \subset V_{\ell-1}\subset V_{\ell+1}\subset \ldots V_{n-1}\subset\mathbb{F}_q^n\},
\]
where $\dim(V_i)=i$. 

We claim that $(1+q)|\alpha_1^{-1}(V_\bullet)|=q|\alpha_0^{-1}(V_\bullet)|+|\alpha_2^{-1}(V_\bullet)|$ for every flag 
\[
V_\bullet = V_1\subset \ldots \subset V_{\ell - 1} \subset V_{\ell+1}\subset \ldots, \subset V_{n-1}\subset \mathbb{F}_q^n.
\]

We have three cases to check, corresponding to whether the fiber  $\alpha_1^{-1}(V_\bullet)$ is empty, a point, or isomorphic to $\mathbb{P}^1_{\mathbb{F}_q}$.
\begin{itemize}
    \item  In the first case there exists $j\in [n]\setminus\{i\}$ such that $TV_j\not\subseteq V_{\mm_1(j)}$. In particular $\alpha_r^{-1}(V_\bullet)=\emptyset$ for $r=0,1, 2$. Remember that $\mm_r(j)=\mm_1(j)$ for every $r=0,2$ and $j\in[n]\setminus\{i\}$.

    \item In the second case, we have $TV_j\subseteq V_{\mm_1(j)}$ for every $j\in [n]\setminus\{i\}$ and $TV_i\not\subseteq V_{\ell-1}$. In this case we have $|\alpha_1^{-1}(V_\bullet)|=1$, $\alpha_0^{-1}(V_\bullet)=\emptyset$ and $\alpha_2^{-1}(V_\bullet)=\mathbb{P}^1_{\mathbb{F}_q}$. 

    \item In the final case, we have $TV_j\subseteq V_{\mm_1(j)}$ for every $j\in [n]\setminus\{i\}$ and $TV_i\subseteq V_{\ell-1}$. In this case $\alpha_0^{-1}(V_\bullet)=\alpha_1^{-1}(V_\bullet) = \alpha_2^{-1}(V_\bullet)=\mathbb{P}^1_{\mathbb{F}_q}$.
    
\end{itemize} 

In all cases we have the relation
\[
(1+q)|\alpha_1^{-1}(V_\bullet)|=q|\alpha_0^{-1}(V_\bullet)|+|\alpha_2^{-1}(V_\bullet)|.
\]
If $\mm_0,\mm_1,\mm_2$ satisfy condition (2) in Definition \ref{def:modular}, the proof is analogous. 
\end{proof}

\begin{theorem}
  For every linear operator $T$ on $\Fq^n$, the number of points on the Hessenberg variety $\he(\mm,T)$ is given by
  $$
|\he(\mm,T)|=\langle F_T(x), \omega X_{G(\mm)}(x;q)\rangle.
$$
\end{theorem}
\begin{proof}
  By Proposition \ref{prop:base} and Lemma \ref{lem:modular}, it follows that
  \begin{align*}
    f(\mm,T)=\sum_{\lambda \vdash n}\frac{a_{q,\lambda}(\mm)}{[\lambda]_q!}f(k_\lambda,T).
  \end{align*}
   For each partition $\lambda=(\lambda_1,\ldots,\lambda_\ell)$,  the quantity $f(k_\lambda,T)$ is equal, by definition, to the number of complete flags $(0)=V_{0}\subseteq V_1\subseteq \cdots \subseteq V_n$ of subspaces of $\Fq^n$ such that $TV_i\subset V_{k_\lambda(i)}$ for each $1\leq i\leq n$ where the tuple $(k_\lambda(1),k_\lambda(2),\ldots,k_\lambda(n))$ is given by
  $$
(\underbrace{\lambda_1,\ldots,\lambda_1}_{\lambda_1 \text{ copies}},\underbrace{\lambda_1+\lambda_2,\ldots,\lambda_1+\lambda_2}_{\lambda_2 \text{ copies}}, \ldots,\underbrace{n,\ldots,n}_{\lambda_\ell \text{ copies}}).
  $$
  If we write $\alpha_i=\lambda_1+\cdots+\lambda_i$, then the containment conditions above are clearly equivalent to the subset of conditions $TV_{\alpha_i}\subseteq V_{\alpha_i}$ for each $1\leq i\leq \ell$. Thus $(0)\subseteq V_{\alpha_1}\subseteq V_{\alpha_2}\subseteq \cdots \subseteq V_{\alpha_\ell}$ is a $\lambda$-flag of $T$ invariant subspaces; it follows from the definition of $F_T(x)$ that the number of choices for such a flag is equal to the coefficient of $m_\lambda$ in $F_T(x)$ which is simply $\langle  F_T(x),h_\lambda\rangle$.

  Once the subspaces $V_{\alpha_i}$ in the complete flag have been chosen, the remaining subspaces $V_j(j\neq \alpha_i)$ can be chosen in $[\lambda]_q!:=[\alpha_1]_q![\alpha_2-\alpha_1]_q![\alpha_3-\alpha_2]_q!\cdots$ ways. It follows that
  \begin{align*}
    f(k_\lambda,T)=[\lambda]_q! \langle  F_T(x),h_\lambda \rangle.
  \end{align*}
  Consequently,
  \begin{align*}
    f(\mm,T)&=\sum_{\lambda \vdash n}a_{q,\lambda}(\mm)\langle  F_T(x),h_\lambda \rangle=\langle F_T(x), \omega X_{G(\mm)}(x;q)\rangle,
  \end{align*}
  proving the result.
\end{proof}

We now prove Corollary \ref{cor:simple_hessenberg} stated in the introduction.

\begin{proof}[Proof of Corollary \ref{cor:simple_hessenberg}]
We begin by recalling the expansion of $X_{G(\mm)}(x;t)$ in the elementary basis when $\mm=(\mm(1), n,\ldots, n)$ (see, for example, \cite[Theorem 4.2]{CH18}):
\[
X_{G(\mm)} =[n-2]!_q \Big(q^{\mm(1)-1}[n-\mm(1)]_qe_{n-1,1} + [n]_q[\mm(1)-1]_qe_n\Big).
\]
On the other hand, by the definition of $F_T(x)$, we have
\[
F_T(x) = m_{n}+ s\; m_{n-1,1} + \cdots,
\]
where $s$ is the number of $1$-dimensional $T$-invariant subspaces. If $T$ has similarity class type $\{(d_1,\lambda^1),\ldots,(d_r,\lambda^r)\},$ then it is easily seen that
\[
s = \sum_{\substack{1\leq i\leq r\\ d_i=1}}[\ell(\lambda^i)]_q,
\]
where $\ell(\lambda)$ denotes the number of parts of $\lambda$. Since $m_\lambda$ and $h_\lambda$ are dual bases with respect to the Hall scalar product, and $\omega e_\lambda=h_\lambda$, it follows that
\[
\he(\mm,T)| = [n-2]!_q\Big([n]_q[\mm(1)-1]_q + q^{\mm(1)-1}[n-\mm(1)]_q\; s\Big).
\]
This completes the proof.    
\end{proof}

\begin{corollary}\label{cor:tripoints}
  If $T$ is triangulable over $\Fq$ of Jordan type $\Lambda=\{\lambda^i\}_{1\leq i\leq r}$, then
  \begin{align*}
   |\he(\mm,T)|=\langle \prod_{i=1}^r \tilde{H}_{\lambda^i}(x;q),\omega X_{G(\mm)}(x;q)   \rangle.
  \end{align*}
\end{corollary}
Let $\mu=(\mu_1,\ldots,\mu_\ell)$ be a partition and consider the Hessenberg function $\mm=k_\mu$. In this case, $G(\mm)$ is a disjoint union of $\ell$ cliques (complete graphs) whose sizes are given by the $\mu_i(1\leq i\leq \ell)$. Therefore $X_{G(\mm)}(x;t)=( \prod_{i=1}^\ell[\mu_i]_t! ) e_\mu$. For this Hessenberg function $\mm$ and with $r=1$ in Corollary \ref{cor:tripoints} we essentially obtain the result stated in Mellit \cite[Cor. 2.13]{MR4125451} and Ram \cite[Thm.~3.2]{ram2024lusztigvarietiesmacdonaldpolynomials}.

\begin{corollary}\label{cor:regularsemisimplepoints}
  If $T$ is regular semisimple, then
\begin{align*}
   |\he(\mm,T)|=\langle  p_\lambda,\omega X_{G(\mm)}(x;q)   \rangle,
  \end{align*}
  where $\lambda$ denotes the partition corresponding to the multiset of degrees of the distinct irreducible polynomials dividing the characteristic polynomial of $T.$
\end{corollary}

Athanasiadis \cite[Thm. 3.1]{MR3359910} proved a combinatorial formula, originally conjectured by Shareshian and Wachs \cite[Conj. 7.6]{MR3488041}, for the coefficients in the power sum expansion of $\omega X_{G(\mm)}(x;t)$. This formula is expressed as a sum of a suitably defined statistic over a class of permutations in the symmetric group $\mathfrak{S}_n$. In view of Corollary \ref{cor:regularsemisimplepoints} this formula corresponds to the number of points on regular semisimple varieties.

 In the remainder of this section $F$ denotes the field $\CC$ of complex numbers or a finite field $\Fq$. Given a Hessenberg function $\mm:[n]\to [n]$ and a triangulable linear operator $X$ on $F^n$ let $\he(\mm,X)$ denote the corresponding Hessenberg variety. Note that every linear operator on $\CC^n$ is triangulable. Tymoczko \cite{MR2275912} proved that the variety $\he(\mm,X)$ admits an affine paving (cellular decomposition). Although this result is proved in the setting of complex Hessenberg varieties, it is easily verified that it carries over to the case of finite fields provided the operator $X$ is triangulable. The existence of an affine paving implies that the cohomology of a complex Hessenberg variety is concentrated in even dimensions. Thus its (modified) Poincaré polynomial is given by
\begin{align*}
  {\rm Poin}(\he(\mm,X);t)=\sum_{d=0}^\delta\beta_{2d} t^d,
\end{align*}
where $\beta_{2d}:=\dim H^{2d}(\he(\mm,X);\QQ)$ denotes the $2d$-th Betti number of the variety while $\delta$ denotes its complex dimension. The Betti number $\beta_{2d}$ can also be interpreted combinatorially as the number of $d$-dimensional cells in an affine paving of $\he(\mm,X)$. For a triangulable operator $X$ on $F^n$, the dimensions and multiplicities of cells in an affine paving depend solely on the Hessenberg function $\mm$ and the Jordan type of $X$ (see Tymoczko \cite[Thm. 7.1]{MR2275912}). Consequently, we have the following result.
\begin{theorem}\label{thm:sameforcandfq}
If $T$ is a triangulable operator on $\Fq^n$, then the number of $\Fq$-rational points on $\he(\mm,T)$ is given by $  {\rm Poin}(\he(\mm,X);q)$ where $X$ is any linear operator on $\CC^n$ with the same Jordan type as $T.$
\end{theorem}

\begin{corollary}\label{cor:poinc}
For each linear operator $X$ on $\CC^n$ of Jordan type $\Lambda=\{\lambda^i\}_{1\leq i\leq r}$, the Poincaré polynomial of the variety $\he(\mm,X)$ is given by
  \begin{align*}
    {\rm Poin}(\he(\mm,X);t)=\langle \prod_{i=1}^r \tilde{H}_{\lambda^i}(x;t),\omega X_{G(\mm)}(x;t)   \rangle.
  \end{align*}
\end{corollary}
\begin{proof}
  Follows from Corollary \ref{cor:tripoints} and Theorem \ref{thm:sameforcandfq} since a polynomial is uniquely determined by the values it takes on any infinite set such as prime powers. 
\end{proof}

 An elegant combinatorial formula for the Poincaré polynomial ${\rm Poin}(\he(\mm,X);t)$ appears in the work of Tymoczko \cite[Thm. 7.1]{MR2275912}. Corollary \ref{cor:poinc} provides a convenient way to compute these polynomials using a computer algebra system.
\begin{problem}
Shareshian and Wachs \cite[Thm. 6.3]{MR3488041} gave a combinatorial formula for the Schur coefficients of $\omega X_{\sf m}(x;t)$. Together with the expansion $\tilde{H}_\lambda(x;t)=\sum_\mu \tilde{K}_{\mu\lambda}(t)s_\mu$, can one obtain a formula for the dimension of the nilpotent Hessenberg variety? (This was posed as an open problem by Tymoczko).
\end{problem}

Specializing Corollary \ref{cor:poinc} to regular operators, we obtain the following result of Brosnan and Chow \cite[Thm. 35]{MR3783432}.
\begin{corollary}\label{cor:poinreg}
  If $X$ is a regular operator on $\CC^n$, then
  \begin{align*}
        {\rm Poin}(\he(\mm,X);t)=\langle h_\mu,\omega X_{G(\mm)}(x;t)   \rangle,
  \end{align*}
  where $\mu$ is the partition whose parts correspond to the sizes of the Jordan blocks in the Jordan canonical form of $X.$
\end{corollary}
An interesting consequence of Corollary \ref{cor:poinreg} is that the Betti numbers of complex regular Hessenberg varieties form a palindromic sequence even though these varieties are not smooth (Brosnan and Chow~\cite[Cor. 36]{MR3783432}). 

\begin{problem}
Find a combinatorial formula for the number of points on an arbitrary Hessenberg variety over $\Fq$.
\end{problem}

\section{Appendix. Combinatorial formula for modified Hall-Littlewood polynomials} 
\label{sec:combform}
\addcontentsline{toc}{section}{Appendix. Combinatorial formula for modified Hall-Littlewood polynomials}
Consider the monomial expansion of the modified Hall-Littlewood polynomial,
\begin{align*}
  \tilde{H}_\lambda(x;t)=\sum_{\lambda}a_{\lambda\mu}(t)m_\mu.
\end{align*}
Since $\tilde{H}_\lambda(x;t)$ is Schur positive, $a_{\lambda\mu}(t)$ is a polynomial with nonnegative integer coefficients. In view of the specialization $\tilde{H}_\lambda(x;1)=h_\lambda$, it follows that $a_{\lambda\mu}(1)$ is equal to the number of nonnegative integer matrices whose row and column sums are given by $\lambda$ and $\mu$ respectively (Stanley \cite[Prop. 7.5.1]{MR1676282}). Butler gave a statistic (see Definition \ref{def:value}) on tabloids (which are in bijection with nonnegative integer matrices) of shape $\lambda$ and content $\mu$ that generates the polynomial $a_{\lambda\mu}(t)$. For partitions $\lambda$ and $\mu$, a tabloid of shape $\lambda$ and content $\mu$ is a filling of the Young diagram ${\rm dg}(\lambda)$ of $\lambda$ (in English notation) which contains $\mu_i$ entries equal to $i$ such that the rows are weakly increasing from left to right (with no condition on the columns; see Figure \ref{fig:tabloid}). In this section we use the expression obtained earlier  for the Poincaré polynomial to give a new derivation of Butler's statistic.

\begin{definition}
Suppose $\lambda$ is a partition of $n$ and let $\mm:[n]\to [n]$ be a Hessenberg function. An $\mm$-\emph{compatible} filling is a filling of ${\rm dg}(\lambda)$ containing each element of $[n]$ such that the configuration \begin{tabular}{|c|c|} \hline $k$ & $j$ \\ \hline \end{tabular} occurs only if $k\leq \mm(j)$.  
\end{definition}
 For an $\mm$-compatible filling $\varphi$, let $v(\varphi)$ denote the number of pairs of entries $i>k$ such that
  \begin{enumerate}
  \item   the entry $i$ appears either below $k$ in the same column or in a column to the left of $k$ (see Figure \ref{fig:ipos}),
  \item if there is a box immediately to the right of $k$ that is filled by $j$, then $i\leq \mm(j)$.
  \end{enumerate}

The following result is a consequence of Corollary \ref{cor:tripoints} and the combinatorial formula of Tymoczko \cite[Thm. 1.1]{MR2275912} for the Poincaré polynomial of a nilpotent Hessenberg variety.

  \begin{theorem}\label{thm:nilpotentstat}
For each partition $\lambda$ of $n$ and Hessenberg function  $\mm:[n]\to [n]$, the Poincaré polynomial of a nilpotent operator $X$ on  $\CC^n$ with Jordan form partition $\lambda$ is given by
  \begin{align*}
    \langle  \tilde{H}_\lambda(x;t),\omega X_{G(\mm)}(x;t) \rangle=\sum_{\varphi}t^{v(\varphi)},
  \end{align*}
where the sum is taken over all $\mm$-compatible fillings $\varphi$ of ${\rm dg}(\lambda)$ containing all positive integers in $[n]$.
\end{theorem}

\begin{definition}
For a tabloid $\theta$, let $w(\theta)$ denote the number of pairs of entries $i>k$ such that
\begin{enumerate}
\item the entry $i$ appears either below $k$ in the same column or in a column to the left of $k$,
\item if there is a box immediately to the right of $k$ that is filled by $j$, then $i\leq j$.
\end{enumerate}
\end{definition}

\begin{theorem} We have
  $$a_{\lambda\mu}(t)=\sum_{\theta} t^{w(\theta)},$$
where the sum is over all tabloids $\theta$ of shape $\lambda$ and content $\mu$. 
\end{theorem}
\begin{proof}
Let $\mu=(\mu_1,\ldots,\mu_\ell)$ and consider the Hessenberg function $\mm=k_\mu$. In this case, $X_{G(\mm)}(x;t)=( \prod_{i=1}^\ell[\mu_i]_t! ) e_\mu$. Since $\omega e_\mu=h_\mu$, it follows by Theorem~\ref{thm:nilpotentstat} that
\begin{equation}\label{eq:a=b}
\langle \tilde{H}_\lambda(x;t),h_\mu \rangle    \prod_{i=1}^\ell[\mu_i]_t!=\sum_{\varphi} t^{v(\varphi)}, 
\end{equation}
where the sum is over all $\mm$-compatible fillings of ${\rm dg}(\lambda)$. Consider the set partition $\mathcal{A}_\mu=\bigcup_{i=1}^\ell A_i$ of $[n]$ where the $i$th block $A_i$ contains all positive integers $x$ satisfying $\mu_1+\cdots+\mu_{i-1}< x\leq \mu_1+\cdots+\mu_i$; thus $A_i=\mm^{-1}(\{\mu_1+\cdots+\mu_i\})$. A filling is $\mm$-compatible precisely when, for each configuration \begin{tabular}{|c|c|} \hline $k$ & $j$ \\ \hline \end{tabular} of adjacent cells, the element $j$ lies in a block with index at least equal to the index of the block containing $k.$

Note that for any ${\mm}$-compatible filling, the filling obtained by arbitrarily permuting all entries in the same block of $\mathcal{A}_\mu$ is also $\mm$-compatible. For a filling $\varphi$, write $v(\varphi)=v_s(\varphi)+v_d(\varphi)$, where $v_s(\varphi)$ is the contribution to $v(\varphi)$ from elements $i>k$ both in the same block of $\mathcal{A}_\mu$ while $v_d(\varphi)$ is the contribution to $v(\varphi)$ from elements $i>k$ in different blocks of $\mathcal{A}_\mu$. If $\bar{\varphi}$ denotes the filling obtained from $\varphi$ by replacing all entries in the block $A_i$ by $i$, then $\bar{\varphi}$ has weakly increasing rows (thus it is a tabloid) and $v_d(\varphi)=w(\bar{\varphi})$. Given a tabloid $\theta$ of shape $\lambda$ and content $\mu$, it can be seen that the contribution to the sum $\sum_\varphi t^{v(\varphi)}$ from all fillings $\varphi$ for which $\bar{\varphi}=\theta$ is given by  $t^{w(\theta)}\prod_{i=1}^\ell[\mu_i]_t!$ (here we have used the fact that $[k]_t!$ is the generating polynomial for the inversion statistic on the set of all permutations of $k$ distinct letters). It follows from \eqref{eq:a=b} that
$$a_{\lambda \mu}(t)=\langle \tilde{H}_\lambda(x;t),h_\mu \rangle=\sum_{\theta} t^{w(\theta)},$$
where the sum is over all tabloids $\theta$ of shape $\lambda$ and content $\mu$.
\end{proof}

  \begin{figure}[htbp]
  \centering
  \begin{subfigure}[b]{0.48\textwidth}
    \centering
        \ytableausetup{centertableaux, boxsize=1em} 
    \begin{ytableau}
      1 &1 &1 &2 &2 &3\\
      1 &3 &3 &3\\
      2 &3 &3\\
      1 &2 &2\\
      1 &1 &1\\
      1 &2\\
    \end{ytableau}
    \caption{
        A tabloid of shape $(6,4,3,3,3,2)$ and content $(9,6,6)$. The entry 3 in the third row and  second column has value 3.
    }
    \label{fig:tabloid}
  \end{subfigure}
  \hfill 
  \begin{subfigure}[b]{0.48\textwidth} 
    \centering
  \includegraphics[scale=.89]{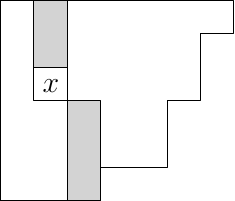}   
  \caption{Entries contributing to the value of  $x$ in a tabloid.}
  \label{fig:valuedef}
  \end{subfigure}
  \caption{Tabloid and value of an element. }
  \label{fig:combined1}
\end{figure} 
 
We now show that $w(\theta)$ is identical to the following statistic of Butler~\cite[Def.~1.3.1]{MR1223236}.
\begin{definition}\label{def:value} 
  If $\theta$ is a tabloid, then the value of an entry $x$ in $\theta$ is the number of smaller entries in the same column and above $x$ or in the next column to the right and below $x$ (see Figure \ref{fig:valuedef}). The value of $\theta$, denoted ${\rm val}(\theta)$, is the sum of the values of the entries in $\theta$.  
\end{definition}
  \begin{figure}[htbp]
  \centering
  \begin{subfigure}[b]{0.45\textwidth}
    \centering
    \includegraphics[scale=0.8]{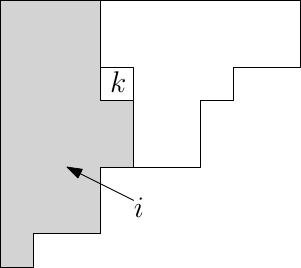}
    \caption{For an entry $k$, possible positions for $i$ which may contribute to $v(\varphi)$ for a filling $\varphi$. }
    \label{fig:ipos}
  \end{subfigure}
  \hfill
  \begin{subfigure}[b]{0.45\textwidth}
    \centering
    \includegraphics[scale=0.8]{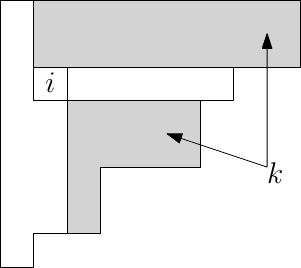}
    \caption{Given $i$, possible positions for $k$ which may contribute to $w(\theta)$ for a tabloid $\theta.$}
    \label{fig:value}
  \end{subfigure} 
  \caption{Overview of possible positions for $i$ and $k$.}      
  \label{fig:combined2}
\end{figure}
   
\begin{theorem}
 We have $w(\theta)={\rm val}(\theta)$ for each tabloid $\theta$.  
\end{theorem}
\begin{proof}
We count pairs $i>k$ that contribute to $w(\theta)$ as follows. Fix an entry $i$ in $\theta$, and consider all possible entries $k<i$ in $\theta$ such that the pair $(i,k)$ is counted by $w(\theta)$. Such an entry $k$ appears either above $i$ and in the same column or in some column to the right of $i$ (see Figure \ref{fig:value}). In particular, $k$ necessarily lies in a row different from $i$. In each row of $\theta$, there can be at most one such element $k$ since, if $j$ appears in the box immediately to the right of $k$, we must have $k<i\leq j$. Moreover, since the rows are weakly increasing, one can determine if a given row $\rho$ contributes 1 to $w(\theta)$ as follows:
  \begin{enumerate}
  \item If $\rho$ lies above $i$, then $\rho$ contributes 1 to $w(\theta)$ precisely when the element in $\rho$ which lies in the same column as $i$ is less than $i$.
  \item If $\rho$ lies below $i$, then $\rho$ contributes 1 to $w(\theta)$ whenever the element of $\rho$ which lies in the column immediately to the right of $i$ is less than $i$.
  \end{enumerate}
Therefore, for a fixed $i$, the number of entries $k$ for which $(i,k)$ contributes to $w(\theta)$ equals the value of $i$. Summing over all entries $i$, we obtain a total contribution of ${\rm val}(\theta)$ and the theorem follows.
\end{proof}

\section{Acknowledgments} 
The first author gratefully acknowledges support from CNPQ, grant numbers 404747/2023-0 and 308637/2023-2. The second author is grateful for the support from  CNPQ, grant number 404747/2023-0. The third author acknowledges with appreciation support from Indo-Russian project DST/INT/RUS/RSF/P41/2021.
\printbibliography  
\end{document}